\documentclass{amsart}

\usepackage{amsmath}
\usepackage{amsfonts}
\usepackage{amsthm}
\usepackage{amssymb}
\usepackage{graphicx}
\usepackage{subfig}
\usepackage{amsrefs}
\usepackage{hyperref}

\begin{document}

\newcommand{\D}{\mathbb{D}}
\newcommand{\HS}{\mathbb{H}}

\newtheorem{theorem}{Theorem}[section]
\newtheorem{lemma}[theorem]{Lemma}
\newtheorem{corollary}[theorem]{Corollary}

\title{On a Theorem of Peter Scott}

\author{Priyam Patel}
\address{Department of Mathematics, Rutgers University, Piscataway, New Jersey 08854}
\email{priyamp@math.rutgers.edu}
\thanks{The author was supported by the Graduate Assistance in Areas of National Need (GAANN) Fellowship}


\begin{abstract}
 We quantify Peter Scott's Theorem that surface groups are locally extended residually finite (LERF) in terms of geometric data. In the process, we will quantify another result by Scott that any closed geodesic in a surface lifts to an embedded loop in a finite cover.
\end{abstract}

\maketitle
\section{Introduction}
A group $G$ is \emph{residually finite} (RF) if for every nontrivial element $g \in G$, there exists a finite index subgroup $G'$ of $G$ such that $g \notin G'$. A group $G$ is called \emph{locally extended residually finite} (LERF) if for any finitely generated subgroup $S$ of $G$ and any $g \in G -S$, then $G$ has a finite index  subgroup $G'$ which contains $S$ but not $g$. 

In recent years, there has been significant work done in an effort to quantify residual finiteness. In particular, for free groups Buskin \cite{Buskin},  Bou-Rabee \cite{Bou1} and Rivin \cite{Rivin} give upper bounds for the index of $G'$ in terms of the word length of the element $g$. Additionally, in the case of nonabelian free groups Bou-Rabee and McReynolds \cite{Bou2} and Kassabov and Matucci \cite{Kass} give lower bounds for the index of the subgroup $G'$, again in terms of word length.

In \cite{Scott}, Peter Scott shows that surface groups are LERF. The goal of this paper is to improve Scott's results by giving an estimate on the index of the subgroup $G'$ in terms of geometric data. The flavor of this paper is rather different from the work on quantifying residual finiteness cited above, namely in the significant use of hyperbolic geometry to quantify residual finiteness and LERF-ness for surface groups. 

Peter Scott also shows in \cite{Scott} that any closed geodesic in a surface $\Sigma$ lifts to an embedded loop in a finite cover of $\Sigma$. The first result in this paper is the following theorem, which quantifies the above result.

\begin{theorem}
Let $\Sigma$ be a compact surface with or without boundary of negative Euler characteristic. Then there exists a hyperbolic metric on $\Sigma$ so that any closed geodesic of length $\ell$ lifts to an embedded loop in a finite cover whose index is bounded by $16.2 \, \ell$. 
\end{theorem}

The idea for the proof of this theorem came from \cite{Scott}. We tessellate the hyperbolic plane by regular, right-angled pentagons as in \cite{Scott}, which will induce a tessellation on $\Sigma$ and on any cover of $\Sigma$. For any compact subsurface $S$ of a surface tessellated by these pentagons, we are able to estimate the area of the smallest, closed, convex union of pentagons $Y$ containing $S$. Using the upper bound on the area of $Y$ and some other geometric results, we obtain the bound of Theorem 1.1 in Section 5.

Our main result, Theorem 7.1, quantifies Peter Scott's LERF theorem. The statement is fairly complicated, but a special case of it is the RF case stated below. 

\begin{theorem}
Let $\Sigma$ be a compact surface of negative Euler characteristic. There exists a hyperbolic metric on $\Sigma$ so that for any $\alpha \in \pi_1(\Sigma)-\{id\}$, there exists a subgroup $H'$ of $\pi_1(\Sigma)$, such that $\alpha \notin H'$. The index of $H'$ is bounded by $32.3\,\ell$, where $\ell$ is the length of the unique geodesic representative of $\alpha$.
\end{theorem}

The proof of Theorem 1.2 relies on the Poincar\'{e} Polygon Theorem (see Chapter 9 of \cite{Beardon}), which will explain the significance of obtaining the convex space $Y$ described above. The proof of Theorem 7.1 will be a natural extension of the proof of Theorem 1.2.

We should note that in \cite{Scott}, there is a gap in Peter Scott's argument that surface groups are LERF. He addresses and fills in this gap in his paper \cite{Scott2}. In our proof of Theorem 7.1, we will make use of the Neilson convex region of a surface, also called the convex core, as Scott does in \cite{Scott2} to avoid the gap in his original paper.

Throughout the paper, we will make use of several standard results of hyperbolic geometry. See \cite{Hubbard}, \cite{Katok}, \cite{Beardon} and \cite{Casson} for details. 

\vspace{.15 in}
{\bf Acknowledgement.} This work was supported by the Graduate Assistance in Areas of National Need (GAANN) Fellowship. The author would like to sincerely thank Feng Luo for all of his support, insight and invaluable suggestions. The author would also like to thank Peter Scott and Ian Agol for reading early drafts of this paper and for their much appreciated feedback. Additionally, the author would like to thank Justin Bush and David Duncan for useful discussions and for their guidance throughout the writing of this paper. Finally, the author would like to thank the referee for being incredibly thorough, and for helpful comments and suggestions that have improved the quality of this paper.

\section{Preliminaries}

We let $\mathbb{D}$ denote the Poincar\'{e} disc model of hyperbolic 2-space and let $\mathbb{H}$ denote the Poincar\'{e} half-plane model.

Following \cite{Scott}, we let $P \subset \mathbb{D}$ be a regular, right-angled pentagon. Let $\Gamma$ be the group of isometries of $\mathbb{D}$ generated by reflections in the five sides of $P$. By the Poincar\'{e} Polygon Theorem \cite{Beardon}, $P$ is a fundamental domain for the action of $\Gamma$ on $\mathbb{D}$, and the images of $P$ under $\Gamma$ tessellate $\D$. Let $T = \{ gP : g \in \Gamma\}$ be the tessellation.

Let $F = \mathbb{R} P^2 \, \#\, \mathbb{R} P^2 \,\#\, \mathbb{R} P^2$. In his paper \cite{Scott}, Scott shows that there exists a fundamental domain for the action of $\pi_1(F)$ on $\mathbb{D}$ consisting of four regular, right-angled pentagons whose sides have been identified in such a way that $\pi_1(F) < \Gamma$. Therefore, $F$ can be tiled by these regular, right-angled pentagons.  He then shows that every closed surface $\Sigma$ of negative Euler characteristic covers $F$. That is, there exists a covering map $r: \Sigma \longrightarrow F$ and an induced map $r_*: \pi_1(\Sigma) \longrightarrow \pi_1(F)$ on their fundamental groups.  $r_*$ is injective as an induced map on $\pi_1$ of a covering map, and therefore, $\pi_1(\Sigma) < \pi_1(F) < \Gamma$. This tells us that there exists a fundamental domain for the action of $\pi_1(\Sigma)$ on $\mathbb{D}$ preserving the tiling, and thus, $\Sigma$ can also be tiled by these pentagons. In fact, the argument above shows that any cover of $F$ can be tiled by such pentagons.

Pulling back the metric induced by the tiling on $F$ via the covering map $r$ gives us a hyperbolic metric on $\Sigma$, which we will call the \emph{standard metric} throughout the paper. All of our results will be for surfaces endowed with this standard metric. The fact that the standard metric is a hyperbolic metric is key. In his paper \cite{Scott2}, Peter Scott demonstrates the gap in his original paper \cite{Scott} with a counterexample for his argument in the Euclidean case. The special properties of hyperbolic space are precisely what allows his revised argument and all of our arguments to work. 

One should note that for a hyperbolic surface $\Sigma$, $\pi_1(\Sigma)$ acts on $\D$ as the deck transformation group for the universal covering space. Therefore, the elements of $\pi_1(\Sigma)$ are isometries of $\D$. If $\alpha \in \pi_1(\Sigma)$, we make a slight abuse of notation and, as a convention, will call the unique geodesic representative in this homotopy class $\alpha$ as well. 

Let $\Sigma$ be a closed hyperbolic surface, tiled by regular, right-angled pentagons, and let $\alpha \in \pi_1 (\Sigma)-\{id\}$. Let $X$ be the cover of $\Sigma$ corresponding to $\langle \alpha \rangle$, the cyclic subgroup of $\pi_1 (\Sigma)$ generated by $\alpha$. Since $\pi_1 (X) \cong \langle \alpha \rangle \cong \mathbb{Z}$, $X$ must be an open annulus or an open M\"{o}bius band depending on if $\alpha$ is an orientation preserving or orientation reversing hyperbolic isometry. In both cases, there exists a lift of $\alpha$ that is the unique simple closed geodesic in $X$, which we will call $\overline \alpha$.

We then have the following sequence of covering maps:

\begin{center} $\D \stackrel{q}{\longrightarrow} X \stackrel{p}{\longrightarrow} \Sigma \stackrel{r}{\longrightarrow}F$ \end{center}

Since $X$ is a cover of $F$, the argument above shows that $X$ can be tiled by regular, right-angled pentagons.

\section{Convexification}

\noindent {\bf Definition.} Let $N$ be a subsurface of a hyperbolic surface $M$. $N$ is \emph{convex} if  for every path $\gamma \subset N$, the geodesic $\gamma^*$ homotopic  rel endpoints to $\gamma$ is also contained in $N$. $N$ is \emph{locally convex} if each point $x \in N$ has a neighborhood in $N$ isometric to a convex subset of $\HS$.

\vspace{.2 in}

\noindent {\bf Definition.} Let $N$ be a subsurface of a hyperbolic surface $M$ tiled by regular, right-angled pentagons. The \emph{convexification} of $N$ is the smallest, closed, convex union of pentagons in $M$ that contains $N$.

\vspace{.15 in}
Recall that $q$ is the covering map $q: \mathbb{D} \longrightarrow X$. Let $S = q(T)$, so that $S$ consists of the pentagons that tile $X$. Let $S_0 \subset S$ be the union of all pentagons $P_i \in S$ such that $P_i \cap \overline \alpha \neq \varnothing$. We choose the basepoint, $a$, of $\overline \alpha$ to be on a geodesic edge of a pentagon in $S_0$ for reasons that will become obvious later. Our first goal will be to convexify, i.e. obtain the convexification of, $S_0$, which we do by adding pentagons along $\partial S_0$ in order to cure the non-convex portions.We obtain a locally convex subsurface $Y$ consisting of a union of pentagons in our tiling of $X$. In section 8.3 of his notes \cite{Thurston}, Thurston shows that for a complete hyperbolic manifold, local convexity implies convexity and so $Y$ will be the desired convexification of $S_0$. In this section, we aim to prove the following theorem:

\begin{theorem}
We can convexify $S_0$ by adding pentagons in our tiling of $X$, so that any pentagon added has non-empty intersection with $S_0$.
\end{theorem}

Recall that the pentagons of $S$ have all angles equal to $\frac{\pi}{2}$. Therefore, $S_0$ can fail to be convex if three pentagons of $S_0$ form an angle of $\frac{3\pi}{2}$ at a vertex on $\partial S_0$. Points of $\partial S_0$ with interior angles equal to $\pi$ will not be referred to as vertices of $S_0$. Thus, all of the vertices of $S_0$ either have interior angle equal to $\frac{3\pi}{2}$ or $\frac{\pi}{2}$. 

\vspace{.15 in}

\noindent {\bf Definition.} If a vertex, $v$, of $S_0$ has interior angle $\frac{3\pi}{2}$, we will call $v$ a \emph{bad corner}. If a vertex, $v$, of $S_0$ has interior angle $\frac{\pi}{2}$, we will call $v$ a \emph{good corner}. 

\vspace{.15 in}

\noindent{\bf Definition.} Choosing an orientation for each boundary component of $\partial S_0$, we will say that corners of  $\partial S_0$ are \emph{consecutive} if they occur consecutively with respect to the chosen orientation.

\vspace{.15 in}

Our results will be independent of the choice of orientation in this definition. 

\vspace{.15 in}

In order to obtain the convexification we must first understand how bad the boundary components of $S_0$ can be. We quantify how bad the boundary is by how many consecutive bad corners occur along it. 

\begin{lemma}
Two bad corners never occur consecutively on a boundary component of $S_0$.
\end{lemma}

\begin{proof}
Let $\widetilde{\alpha}$ be a lift of $\overline \alpha$ to $\D$. Lift every pentagon of $S_0$ to its lift that intersects $\mathring{\widetilde{\alpha}}$. Doing so, we have lifted all of $S_0$ and $\partial S_0$. We call their lifts $\widetilde{S_0}$ and $\widetilde{\partial S_0}$. Now suppose we have a bad corner in $S_0$ formed by three pentagons $P_1$, $P_2$, $P_3 \in S_0$. We have lifted these three pentagons to $\widetilde P_1$, $\widetilde P_2$ and $\widetilde P_3$ in $\D$. By construction, $\widetilde P_1$, $\widetilde P_2$ and $\widetilde P_3$ intersect $\widetilde{\alpha}$ and form a bad corner, $B_1$, on $\widetilde{\partial S_0}$. Translating by an element of $\Gamma$, we may assume that $\widetilde P_1$, $\widetilde P_2$ and $\widetilde P_3$ form the region in Figure~1(A) below.

\begin{figure}[h]
\centering
\subfloat[]{\label{}\includegraphics[trim = -.5in 2.48in -.5in 2in, clip=true, totalheight=0.2\textheight]{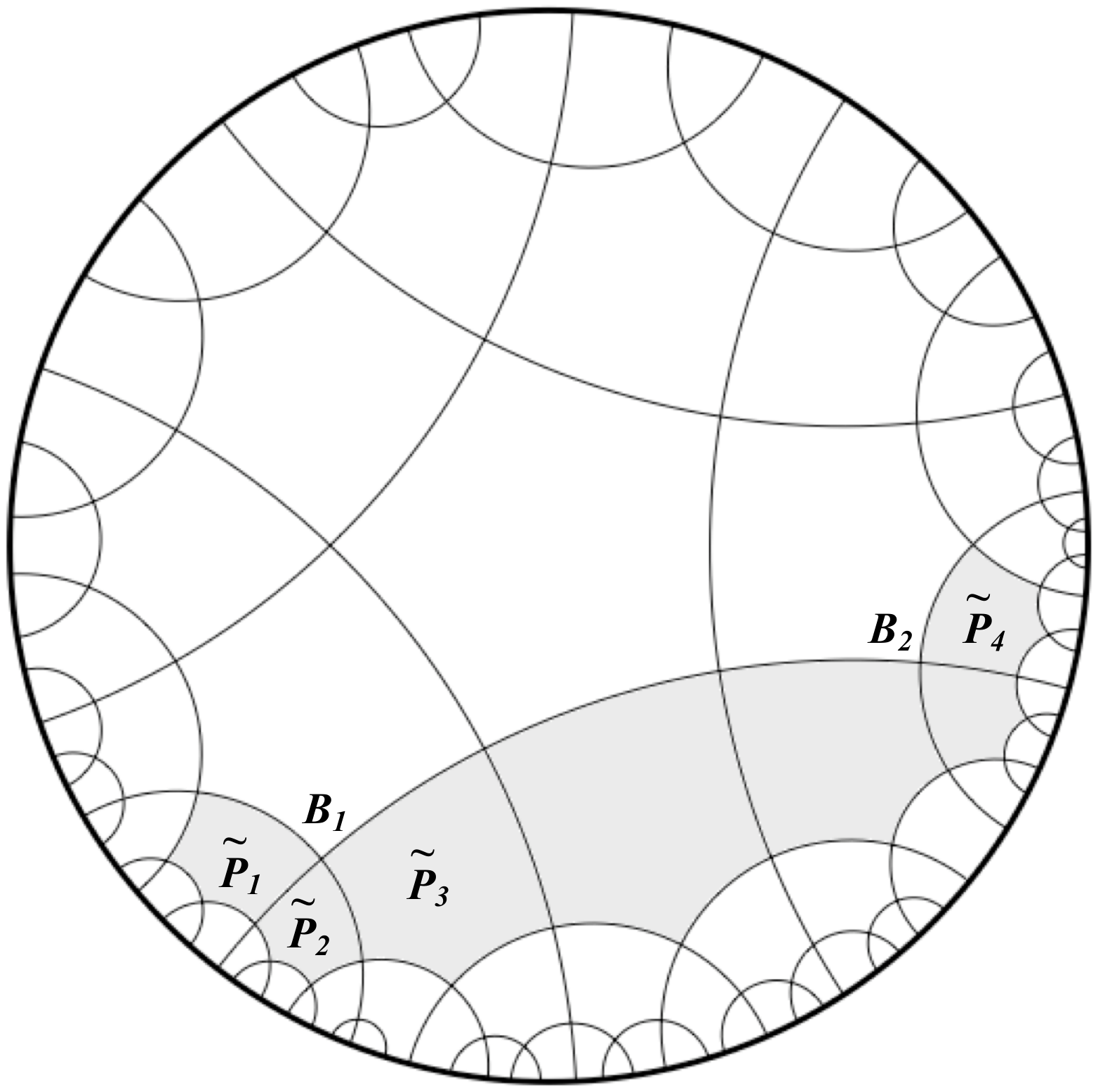}}
\subfloat[]{\label{} \includegraphics[trim = -.5in 2.5in -.5in 2in, clip=true, totalheight=0.2\textheight]{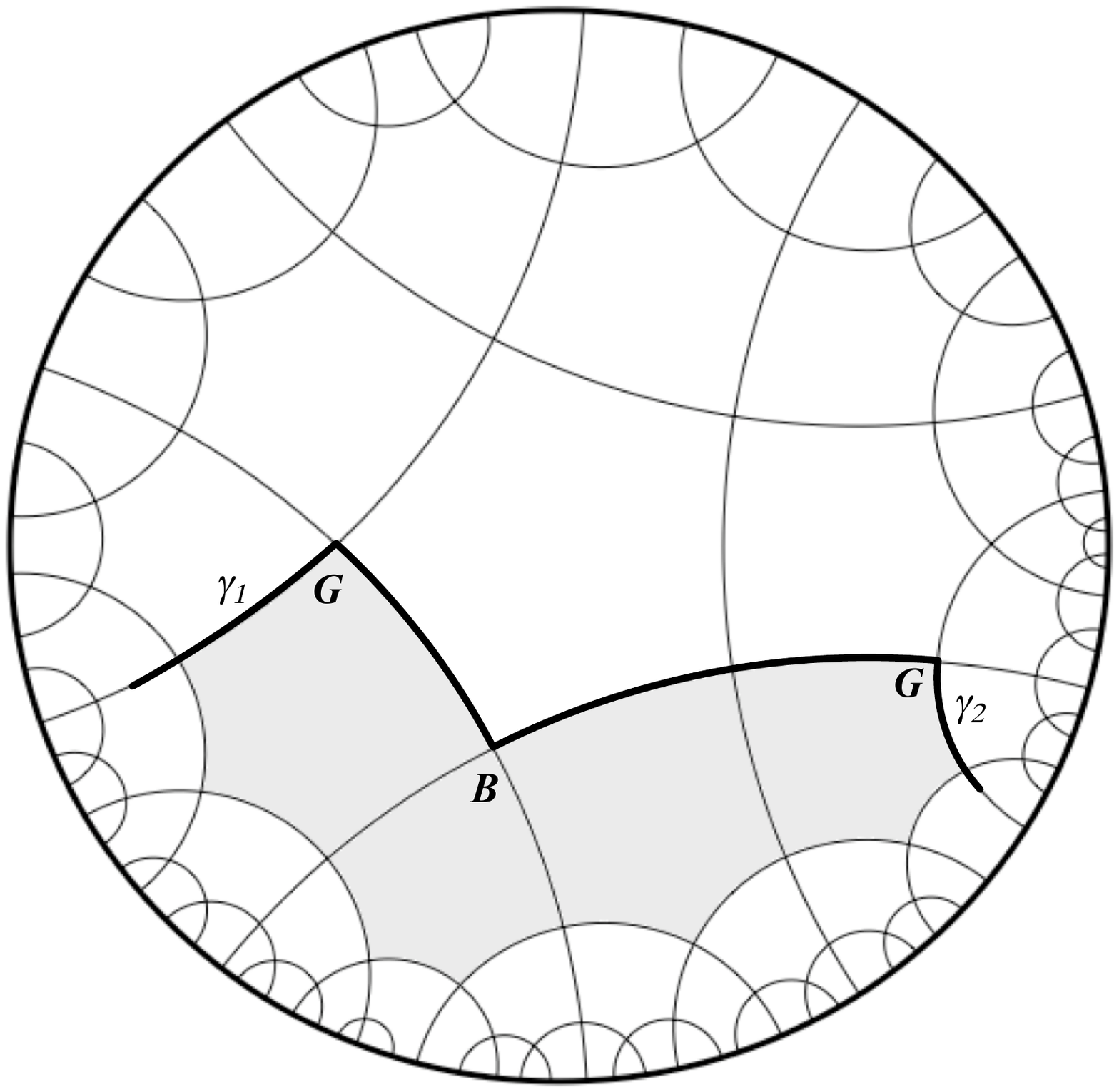}}
\caption{}
\end{figure}

If we hit a second consecutive bad corner, $B_2$, travelling along $\widetilde{\partial S_0}$, then there exists a pentagon $\widetilde P_4 \in \widetilde{S_0}$, as in Figure~1(A) above, that is one of the pentagons that forms $B_2$. Since $\widetilde P_1$, $\widetilde P_4 \in \widetilde{S_0}$,  $\widetilde{\alpha}$ must intersect both of those pentagons, but $\widetilde{\alpha}$ cannot intersect any of the white pentagons between them. There is no such geodesic in $\D$. Thus, there can never be two consecutive bad corners along a boundary component of $S_0$. 

\end{proof}

The result above is independent of the choice of orientation for the boundary component containing $B_1$, so that two consecutive bad corners can never occur along a boundary component of $S_0$ with respect to either of the two choices of orientation for the boundary component.

To each of the boundary components of $S_0$ we can associate a word, $w$, in the letters $G$ and $B$ by reading off whether the consecutive corners are good or bad along the boundary component. For example, if a boundary component contains two bad corners and four good corners the word $w$ could be $w= GBGBGG$. For long words we will write $w = \cdots GBGGBGG \cdots$, by which we mean that we are reading off only a portion of the corners along the boundary component. An immediate consequence of Lemma 3.2 is the following. 

\begin{corollary}
Let $B_0$ be a bad corner of a boundary component of $S_0$. Then the word $w$ associated to that boundary component of $S_0$ must be of the form $w = B_0, w= B_0G, w = GB_0$ or $w = \cdots GB_0G \cdots$.  
\end{corollary}

\begin{proof}
If $B_0$ is the only corner of that boundary component, then $w=B_0$. If the boundary component consists of exactly two corners, Lemma 3.2 tells us that $w \neq B_0B_1$ where $B_1$ is another bad corner since bad corners never occur consecutively along $\partial S_0$.  Thus, $w=B_0G$ or $w=GB_0$. If the boundary component has three or more corners, then again by Lemma 3.2 we know that two bad corners never occur consecutively along $\partial S_0$ regardless of the orientation we choose for the boundary component. Thus, $w \neq \cdots GB_0B_1 \cdots$ and $w \neq \cdots B_1B_0G \cdots$, and $w$ must therefore be of the form $\cdots GB_0G \cdots$. 
\\
\end{proof}

Now we will attempt to convexify $S_0$ by adding pentagons near the bad corners and show that the number of pentagons needed can be bounded. 

\begin{proof}[Proof of Theorem 3.1]
So long as we are not in the case where $w = B$, $w= BG$ or $w=GB$ (see the note below), the picture at each bad corner looks like Figure~2 below. If we extend the far geodesic edges of the good corners and add the pentagons that intersect $S_0$ lying between these two extended geodesic segments, we will have convexified this portion of $S_0$.

\begin{figure}[h]
\centering
 \includegraphics[trim = 1.6in 4.45in 1.6in 4in, clip=true, totalheight=0.11\textheight]{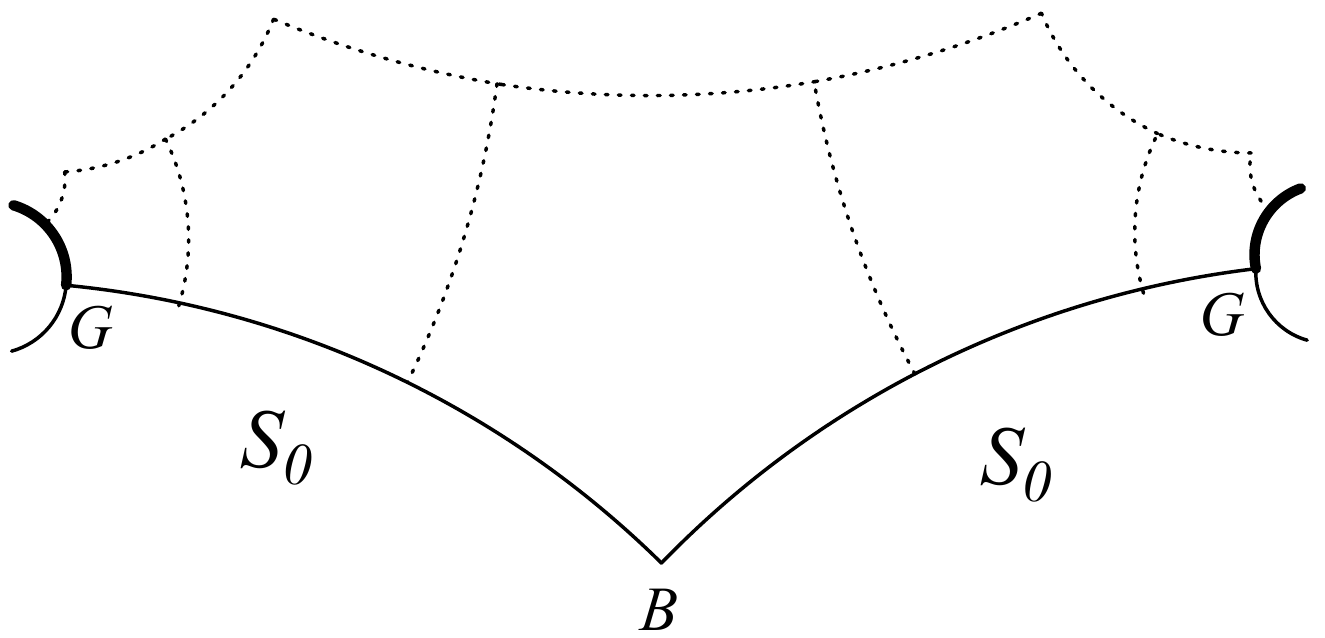}
 \caption{}
\end{figure}

{\bf Note:} If $w = B$, the lift of this component of $\partial S_0$ to $\D$ will look exactly like the figure above. In this case, the geodesics with the bold extensions will have been identified in $S_0$ and we will have convexified this boundary component. In fact, this one step is enough to obtain the convexification for the case where $w = GB$ or $BG$ as well. So in these cases the convexification procedure is complete here.
\\
\\
\indent We follow this procedure for every bad corner of $S_0$.  As shown earlier, the bad corners of $S_0$ are separated by one or more good corners. We will see below that it is slightly easier to convexify portions of boundary components where bad corners are separated by two or more good corners so we will handle this case first. 
\\
\\
\emph{Case 1:} Suppose $B_1$ is a bad corner of $S_0$ and that travelling along $\partial S_0$, the next bad corner we hit, $B_2$, is separated from $B_1$ by two or more good corners. Then Figure~3(A) below shows that we have convexified the region around $B_1$ and $B_2$ and we need not add any more pentagons between the two bad corners.

\begin{figure}[h]
\centering
 \subfloat[]{\label{} \includegraphics[trim = .2in 3.25in .2in 3.15in, clip=true, totalheight=0.14 \textheight]{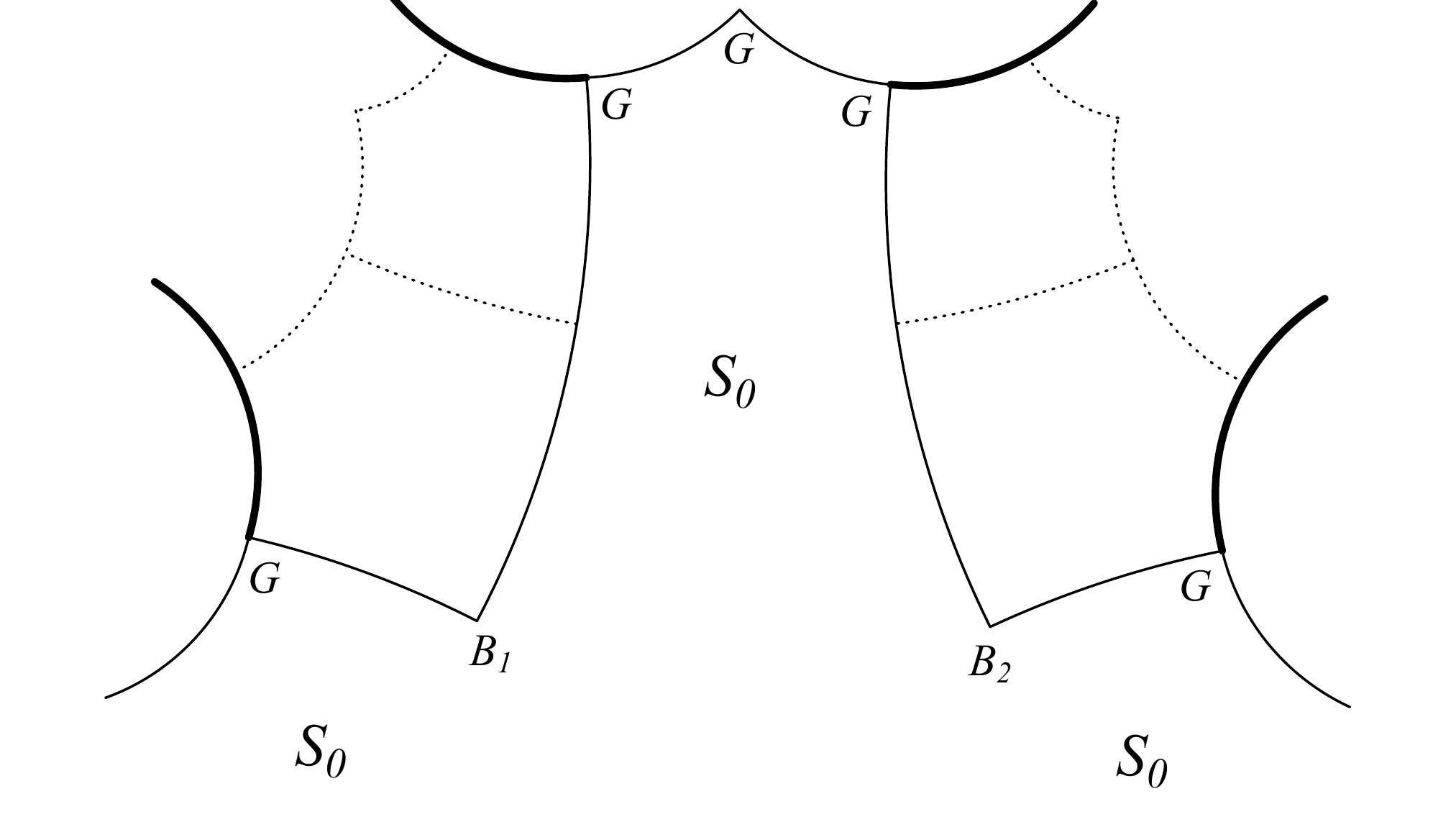}}
 \subfloat[]{\label{}  \includegraphics[trim = .2in 3.3in .2in 3.25in, clip=true, totalheight=0.14\textheight]{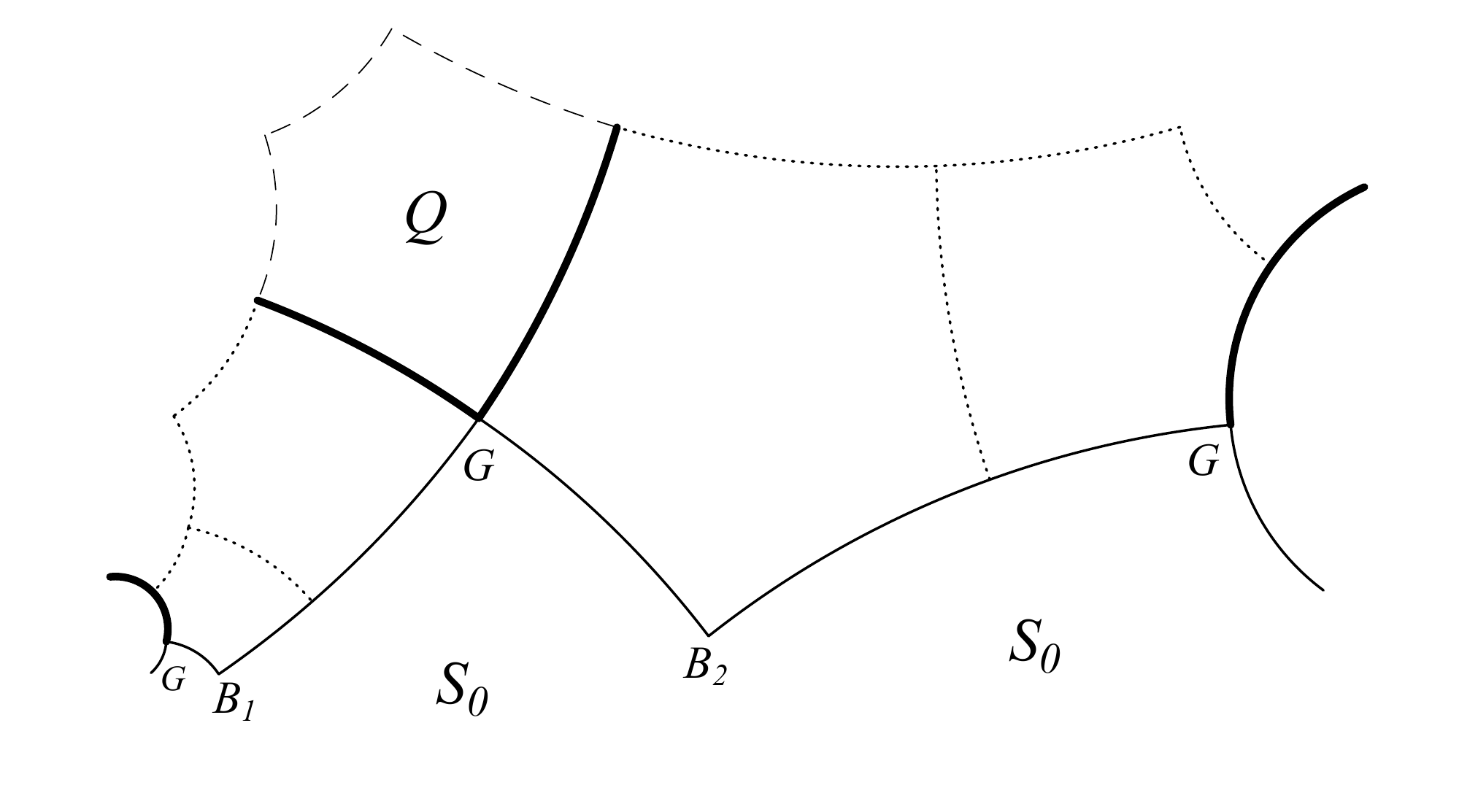}}
 \caption{}
\end{figure}

\emph{Case 2:} Suppose that $B_2$ is separated from $B_1$ by only one good corner, that is $w = \cdots GB_1GB_2G \cdots$. Then we will need to add one more pentagon of $S$ to $S_0$, which will still have non-empty intersection with $S_0$. 

As seen in Figure~3(B) above, we have created a new bad corner during our attempt to convexify $S_0$. However, we can simply add in the one missing pentagon, $Q$, which still intersects $\partial S_0$ at one point. 

After adding all such pentagons $Q$, we have a set $Y$ containing $S_0$ and obtained by adding pentagons that all intersect $S_0$. $Y$ is locally convex, and by the comments above $Y$ is therefore the convexification of $S_0$ we were looking for.
\\
\end{proof}

\section{Bounding the number of pentagons in $Y$}
Our next goal will be to obtain an upper bound on the number of pentagons in $Y$, the convex set obtained in the proof of Theorem 3.1. This bound will play a crucial role in proving our main result, Theorem 7.1.

\begin{lemma}
The diameter of each pentagon in our tiling of $\D$, and therefore in our tiling of $X$, is $\cosh^{-1} \left( \left(1+ 2 \cos \dfrac{2\pi}{5}\right)^2\right)$.
\end{lemma}
\begin{proof}
Each pentagon in our tiling has angles $\frac{\pi}{2}$. For simplicity we will work with a regular, right-angled pentagon $P$ centered at the origin of $\D$. 

The longest geodesic segments between any two points of $P$ are represented by the dotted lines in Figure~4(A). Call these five segments $\gamma_1, \ldots, \gamma_5$ of lengths $\ell_1, \ldots, \ell_5$ respectively. It turns out that $\ell_i$ are all equal. 

\begin{figure}[h]
\centering
 \subfloat[]{\label{}  \includegraphics[trim = .5in 3.8in .5in 3.85in, clip=true, totalheight=0.13\textheight]{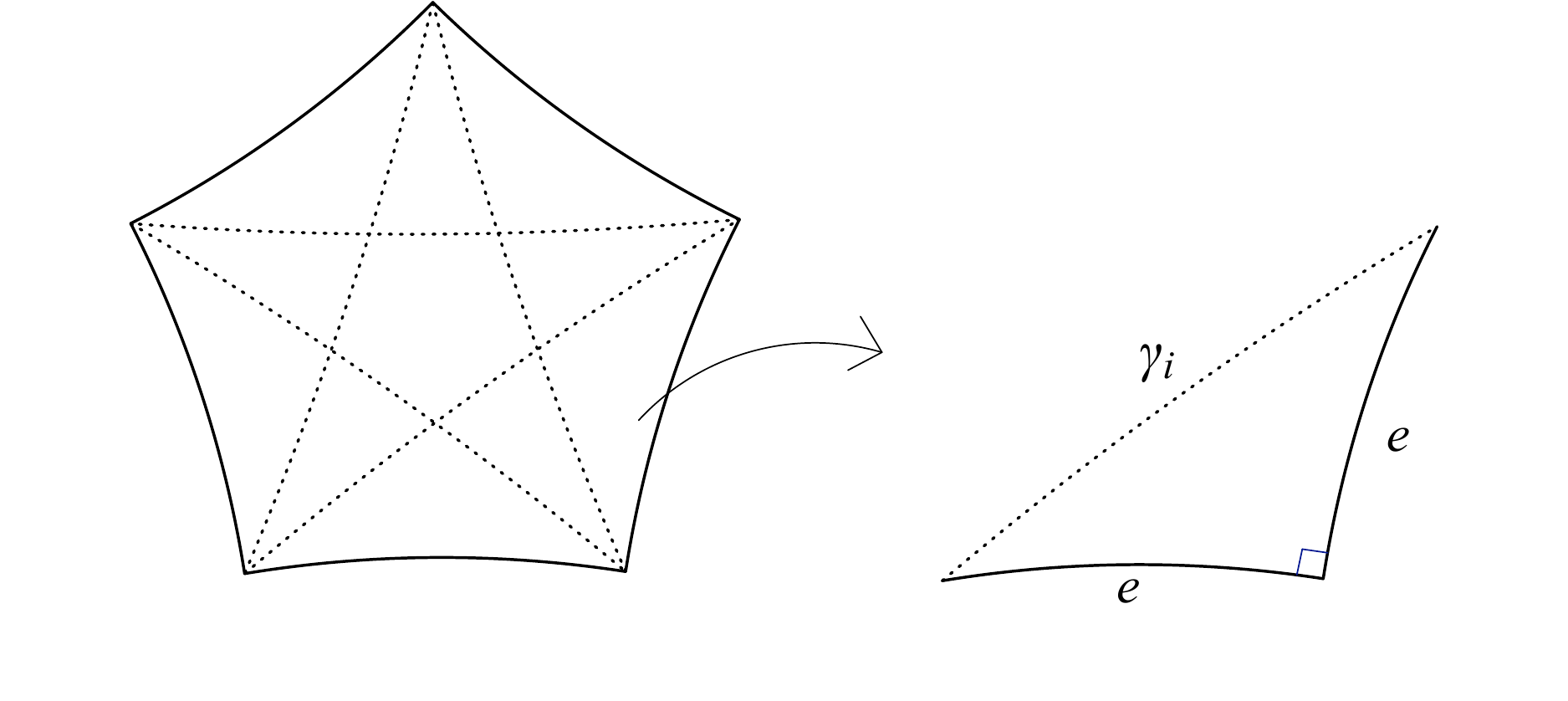}}
  \subfloat[]{\label{}  \includegraphics[trim = .5in 4.2in .5in 3.85in, clip=true, totalheight=0.13\textheight]{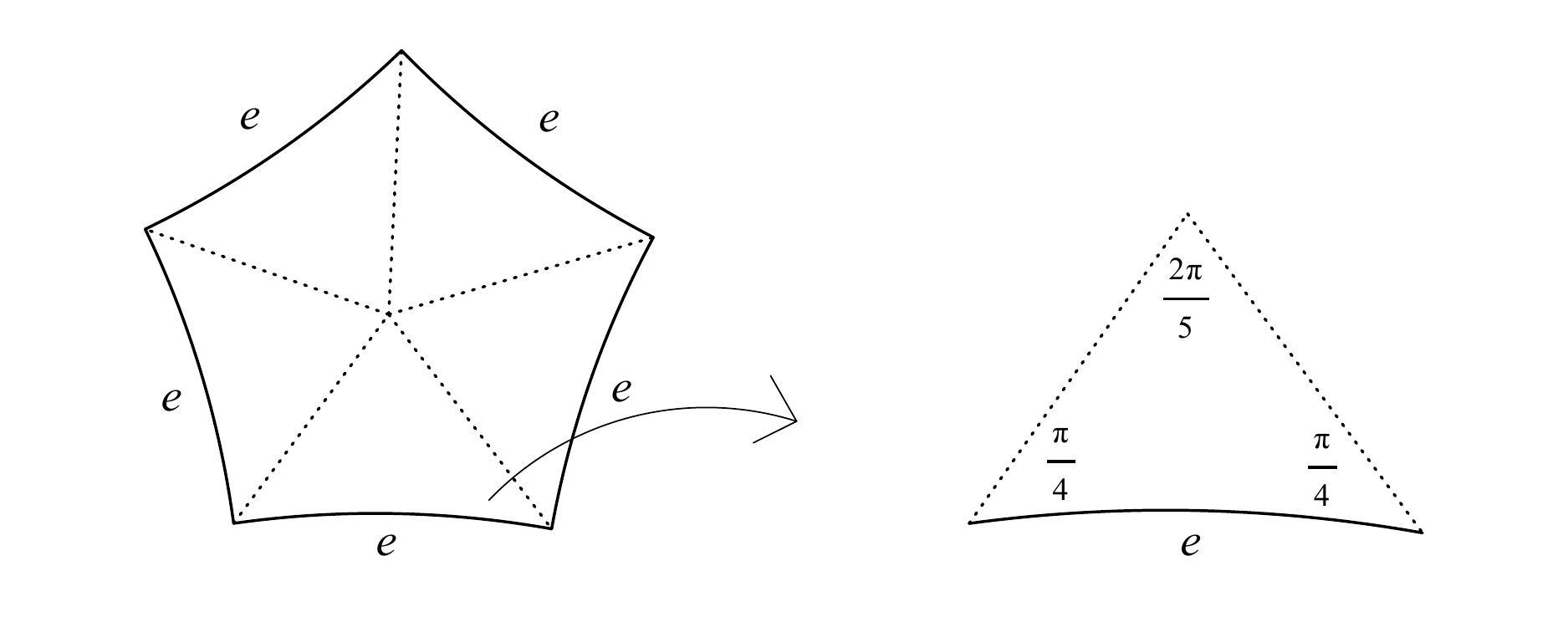}}
  \caption{}
\end{figure}

First we will find the lengths of the sides of $P$, which are also all equal. Break $P$ into the five triangles in Figure~4(B) above. Each triangle has angles $\frac{2\pi}{5}$, $\frac{\pi}{4}$ and $\frac{\pi}{4}$ with the side of $P$, whose length we will call $e_P$, opposite the $\frac{2\pi}{5}$ angle. Using a hyperbolic Law of Cosines \cite{Beardon}, we have that $e_P = \cosh^{-1} \left( 1+ 2 \cos \left( \dfrac{2\pi}{5} \right)\right)$.

Now going back to Figure~4(A), we see that each of the $\gamma_i$  forms a side of a hyperbolic triangle opposite a right angle where the other two sides are of length $e_P$. Using another hyperbolic Law of Cosines \cite{Beardon} we have that $\ell_i = \cosh^{-1} \left( \left( \cosh e_P \right)^2\right) =  \cosh^{-1} \left( \left(1+ 2 \cos \dfrac{2\pi}{5}\right)^2\right) \approx 1.167.$

\end{proof}

The length $\ell_i$ is the diameter of $P$ and we now call this diameter $d_0$. We will use $d_0$ to bound the number of pentagons in $Y$, but first we need the following lemma.

\begin{lemma}
Let $\Omega \subset \HS$ be the region in Figure~5. Then $Area(\Omega) = \ell_0 \sinh b$, where $\ell_0$ is the length of the geodesic segment between $r_0 i$ and $R_0 i$. 
\end{lemma}

\begin{figure}[h]
\centering
 \includegraphics[trim = 1.25in 3in 1.25in 3in, clip=true, totalheight=0.135\textheight]{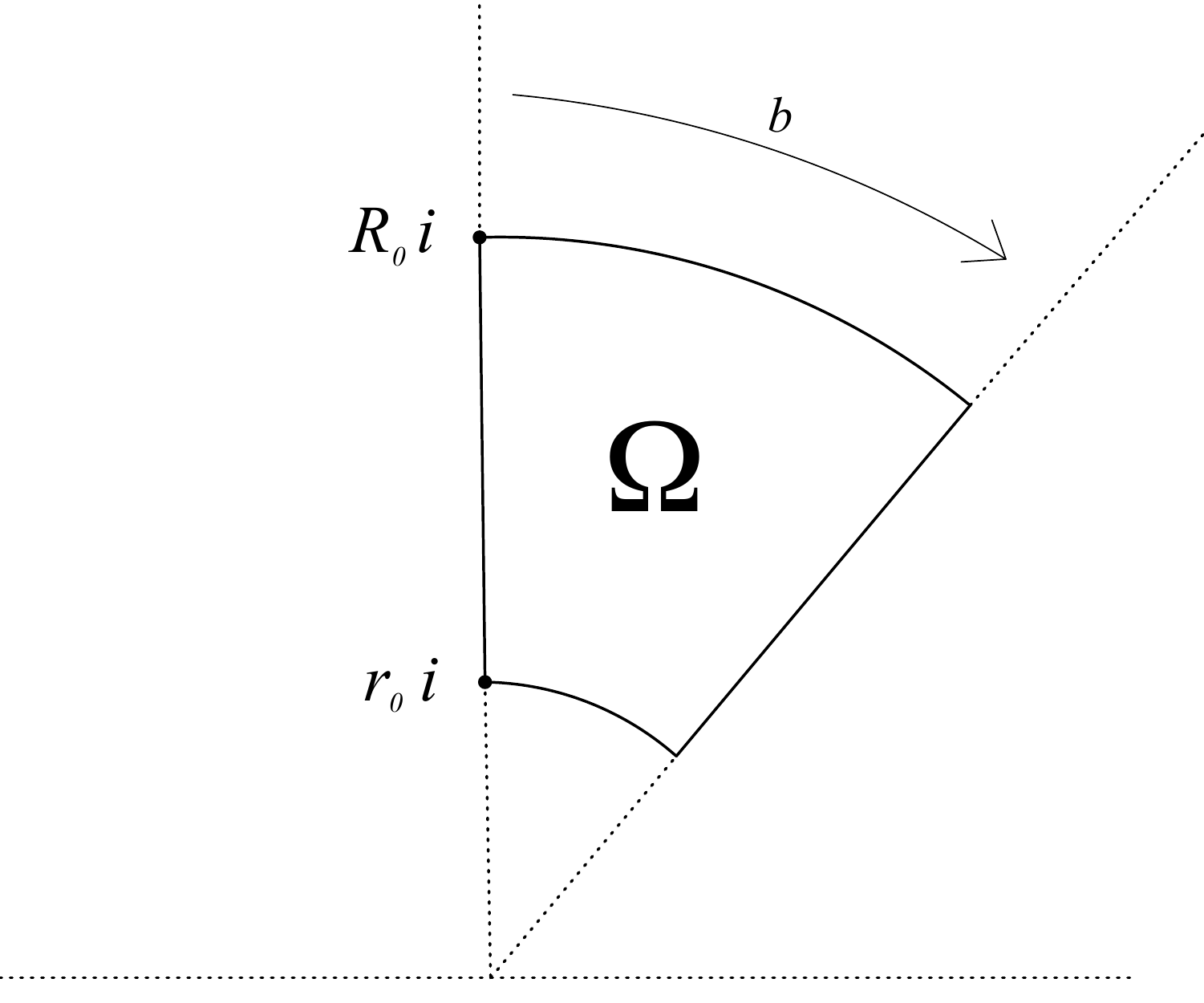}
 \caption{}
\end{figure}

\begin{proof}
$Area(\Omega) = \displaystyle\iint \limits_{\Omega} \frac{\mathrm{d}x \wedge \mathrm{d}y}{y^2} \; \;= \; \; \displaystyle\int_{\frac{\pi}{2} - \theta_0}^{\frac{\pi}{2}} \displaystyle\int_{r_0}^{R_0} \frac{\mathrm{d}r \wedge \mathrm{d}\theta}{r\sin^2 \theta} \; \; = \; \; \ln\left(\frac{R_0}{r_0}\right) \displaystyle\int_{\frac{\pi}{2} - \theta_0}^{\frac{\pi}{2}}\; \frac{1}{\sin^2 \theta}\; \mathrm{d}\theta \; \; = \; \; \ell_0 \left( -\cot \theta \bigg|_{\frac{\pi}{2} - \theta_0}^{\frac{\pi}{2}} \right)\; \; = \; \; \ell_0 \cot \left(\frac{\pi}{2} - \theta_0 \right)$.
 
 \vspace{.2 in}
\noindent By the angle of parallelism laws \cite{Katok}, $\cot \left( \dfrac{\pi}{2} - \theta_0 \right) = \sinh b$, so that $Area(\Omega) = \ell_0 \sinh b$.
\\
\end{proof}

\begin{theorem}
Let $\ell$ be the length of $\overline \alpha$ in $X$, and hence, the length of $\alpha$ in $\Sigma$. Then $Area(Y) \leq 2 \,\ell \sinh (2d_0)$.
\end{theorem}

\begin{proof}
Let $ Z = \left\{ x \in X : d(x, \overline \alpha) \leq 2d_0 \right\}$. We know that every pentagon of $Y$ either intersects $\overline \alpha$ (and is an element of $S_0$) or intersects $S_0$. Thus, $\!\displaystyle \sup_{y \in Y} \left\{ d( y, \overline \alpha) \right\} \leq 2 \, d_0$, so that $Y \subseteq Z$.

Recall that $a$ is the basepoint of $\overline \alpha$. Choose points $z_1$ and $z_2$ on the two different boundary components of $Z$, such that $d( z_i, a) = 2\, d_0$. Let $\beta_i$ be the geodesic segment between $a$ and $z_i$, for $i = 1,2$. Then a lift of $Z$ to $\HS$ looks like the region in Figure~6, where $\widetilde \beta_i$ and $\widetilde \beta_{i}'$ are two lifts of $\beta_i$, for $i = 1,2$, and $\widetilde \alpha$ is a lift of $\overline \alpha$. We say that $Z$ ``opens" along the $\beta_i$.

\begin{figure}[h]
\centering
 \includegraphics[trim = 1.9in 1.5in 2in 1.5in, clip=true, totalheight=0.15\textheight]{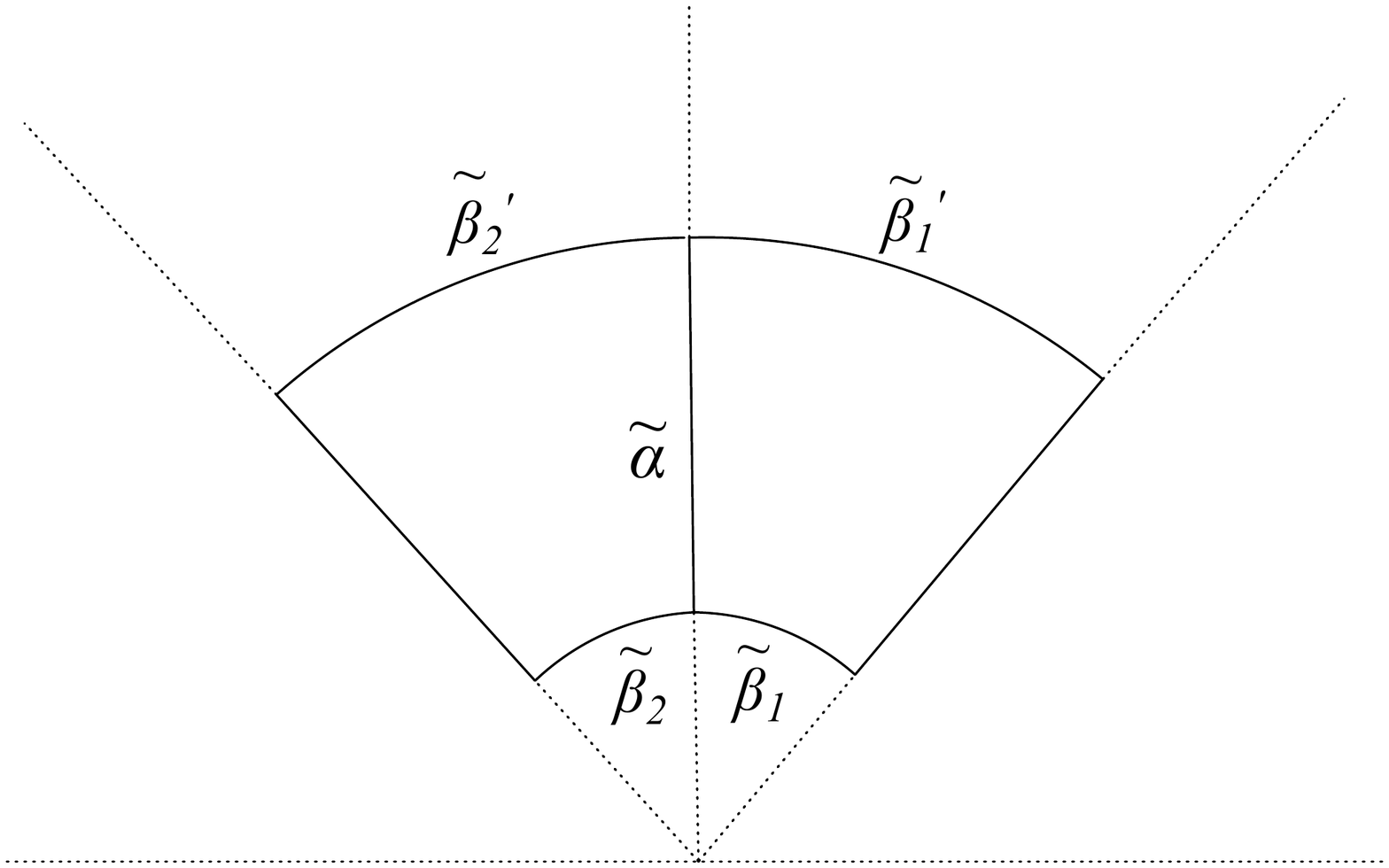}
 \caption{}
\end{figure}

Since the lengths of $\widetilde \beta_1$, $\widetilde \beta_1', \widetilde \beta_2$, $\widetilde \beta_2'$ are all equal to $2 \, d_0$, we can apply Lemma 4.2, which tells us that $Area(Z) = 2 \ell \sinh (2\, d_0)$. Thus, we have $Area(Y) \leq Area(Z) = 2 \ell \sinh (2\,d_0)$. We note that when $\Sigma$ is unorientable, $Z$ has one boundary component, but there exist many choices of $z_1$ and $z_2$ such that $d( z_i, a) = 2\, d_0$ and such that the geodesic $(z_1, z_2)$ between the points is orthogonal to $\partial Z$. Opening $Z$ along such a geodesic yields precisely the region in Figure~6 above. 
\\
\end{proof}

\begin{corollary} If $Y$ consists of $k$ pentagons, then $Area(Y) = k \, \frac{\pi}{2} \leq 2\,\ell \sinh (2 \, d_0)$. Solving for $k$ we have, $k \leq \dfrac{4 \, \sinh (2 \, d_0)}{\pi} \, \ell \approx 16.131\, \ell$. Thus, $k < 16.2 \, \ell$.
\end{corollary}
\qed

\begin{corollary}
If $\overline \alpha$ is the image of a geodesic line in our tessellation of $\D$, then $k < 3.1\, \ell$. 
\end{corollary}

\begin{proof}
Since $Y = S_0$, $Y \subseteq \left\{ x \in X : d(x, \overline \alpha) \leq d_0 \right\}$. Therefore, $k \leq \dfrac{4 \, \sinh (d_0)}{\pi} \, \ell \approx 3.081\, \ell< 3.1\, \ell$. 

\end{proof}

We should note that the bounds given in Corollaries 4.4 and 4.5 are not sharp. By going to the $2d_0$ neighborhood of the geodesic $\overline \alpha$ we are overestimating the area of the convex union of pentagons $Y$, and therefore overestimating the number of pentagons in the set. Also, in the case where $\overline \alpha$ is the image of a geodesic line in our tessellation, we can give the exact number of pentagons in $Y= S_0$ to be $k= \frac{2\ell}{e_P}$ where $e_P$ is the length of the edges in the regular right-angled pentagons computed above. We use the bound in Corollary 4.5 because it is analogous to the bound in Corrollary 4.4 and using this estimate eliminates the need for a separate argument in the proof of Theorem 7.1.

\section{Lift to Finite Cover}
In this section we will prove Theorem 1.1 in the closed case and then show that the argument easily extends to the case of a compact surface with boundary. 

Let $\Sigma$ be a closed surface endowed with the standard metric defined in Section 2. Recall that $X$ is the cover of $\Sigma$ corresponding to $\langle \alpha \rangle$, and that $\D \stackrel{q}{\longrightarrow} X $ is the universal covering map. Therefore, the deck transformation group $\langle \alpha \rangle$ acts by isometries on $\D$. The axis of the isometries of $\langle \alpha \rangle$ is the geodesic line in $\D$ consisting of all of the lifts of $\overline \alpha$. Call this axis $L$.

Let $\widetilde Y$ be the set of all lifts of the pentagons in $Y$ to $\D$, and let the set of isometries of $\D$ consisting of reflections in the sides $y_i$ of $\widetilde Y$ be denoted by $R = \left\{R_{y_i}\!: y_i \;\text{is a side of}\; \widetilde Y \right\}$. Then $\widetilde Y$ is a fundamental domain for the action of $\langle R \rangle$ on $\D$ by the Poincar\'{e} Polygon Theorem \cite{Beardon}. Note that $\langle R \rangle < \Gamma$ since the sides of $\widetilde Y$ are lines in our tessellation of $\D$, and a reflection in any of these lines is an element of $\Gamma$. 

Let $K = \langle R, \alpha \rangle$, and let $\widetilde{\alpha}$ be one lift of $\overline \alpha$ to $\D$. Next, lift each pentagon of $Y$ to one of its lifts so that the result is a connected union of $k$ pentagons in $\D$ containing $\widetilde \alpha$. Call this union of $k$ pentagons $\overline Y$.

\begin{lemma}
$\overline Y$ is a fundamental domain for the action of $K = \langle R, \alpha \rangle$ on $\D$.
\end{lemma}

\begin{proof}
We know that $\underset{\alpha^n \in \langle \alpha \rangle}{\bigcup} \alpha^n\, \overline Y = \widetilde Y$, where $\widetilde Y$ is the set of all lifts of the pentagons of $Y$. Since $\widetilde Y$ is a fundamental domain for $\langle R \rangle$, we also know that $\underset{r \in \langle R \rangle}{\bigcup} r\,\widetilde Y = \D$. Therefore, $\underset{k \in K}{\bigcup} k\,\overline Y = \D$. Now must show that $k \, \mathring{\overline Y} \cap \mathring{\overline Y} = \varnothing$, for all $k \in K- \{id\}$. 

Since $\langle \alpha \rangle$ is the deck transformation group for $\D \stackrel{q}{\longrightarrow} X $ and $\overline Y$ contains only one lift of each pentagon in $Y$, we know that $\alpha^n \mathring{\overline Y} \cap \mathring{\overline Y} = \varnothing$ for all $\alpha^n \in \langle \alpha \rangle - \{id\}$.

Next, recall that $R_{y_i}$ denotes a reflection in the side $y_i$ of $\widetilde Y$. If $\alpha^n y_1 = y_2$ where $y_1$ and $y_2$ are sides of $\widetilde Y$, then we have the relation $\alpha^n R_{y_1} = R_{y_2}\alpha^n$ in $K$, in other words $\alpha^n R_{y_1} \alpha^{-n} = R_{y_2}$. In fact, there exists a group homomorphism $\phi: \langle \alpha \rangle \longrightarrow Aut(\langle R \rangle)$, defined by $\phi(\alpha^n) \rightarrow \phi_{\alpha^n}$, where $\phi_{\alpha^n} (r) = \alpha^n r \alpha^{-n}$ for all $r \in \langle R \rangle$ and $\alpha^n \in \langle \alpha \rangle$. Therefore, $K = \langle R \rangle \rtimes_{\phi} \langle \alpha \rangle$, and it follows that every element of $K$ can be written as $r\alpha^n$ where $r \in \langle R \rangle$. 

Now, $\widetilde Y$ is a fundamental domain for $\langle R \rangle$, so $r \, \mathring{\widetilde Y} \cap \mathring{\widetilde Y} = \varnothing$ for every $r \! \in \!\langle R \rangle - \{id\}$. Let $k = r\alpha^n \in K - \{id\}$. If $r \neq id$ then $\alpha^n \mathring{\overline Y} \subset \mathring{\widetilde Y}$ and $r\alpha^n \mathring{\overline Y} \cap \mathring{\widetilde Y} = \varnothing$, and hence $k \mathring{\overline Y} \cap \mathring{\overline Y} = \varnothing$. If $ r = id$ then $\alpha^n \neq id$ and we have shown above that in this case $k\mathring{\overline Y} \cap \mathring{\overline Y} = \alpha^n \mathring{\overline Y} \cap \mathring{\overline Y} = \varnothing$.

\end{proof}

\begin{lemma}
Let $K' = K \cap \pi_1 (\Sigma)$. Then, $[\pi_1(\Sigma)\!:\!K'] \leq [\, \Gamma\!:\!K] = k$.
\end{lemma}

\begin{proof}
Since $\overline Y$ is a fundamental domain for $K$ and consists of $k$ pentagons, we know that $[\, \Gamma\!:\!K] = k$. Poincar\'{e}'s Theorem \cite{Jacobson} states that if $H_1$ and $H_2$ are subgroups of a group $G$, then $[H_1\!:\!H_1 \cap H_2] \leq [G\!:\!H_2] $. The lemma follows by letting $G = \Gamma$, $H_1 = \pi_1(\Sigma)$ and $H_2 = K$ in Poincar\'{e}'s Theorem. 

\end{proof}

Next, we let $N$ be the cover of $\Sigma$ corresponding to the subgroup $K' = K \cap \pi_1 (\Sigma)$ of $\pi_1 (\Sigma)$. That is $ N = \D/ K'$. Let $s: \D \longrightarrow N$ be the covering map. 

\begin{lemma}
The image of $\widetilde \alpha$ under $s$ is an embedded loop in N.
\end{lemma}

\begin{proof}
Since $K'$ is a subgroup of $K$, the covering map $f: \D \longrightarrow \D/K$ factors as $f = u \circ s$ where $s$ and $u$ are the covering maps in the sequence $\D \stackrel{s}{\longrightarrow} \D/K' \stackrel{u}{\longrightarrow} \D/K$. Let $\alpha' = s(\widetilde \alpha)$. We will first show that $f(\widetilde \alpha)$ is an embedded loop in $\mathbb{D}/K$, and then use this fact to show that $\alpha'$ is an embedded loop in $N = \mathbb{D}/K'$.

 Since $\alpha \in K$ we know that $f(\widetilde \alpha)$ is certainly a loop in $\mathbb{D}/K$. By Lemma 5.1 the set $\overline Y$, defined above, is a fundamental domain for the action of $K$ on $\mathbb{D}$. Recall that $\overline \alpha$ is the lift of $\alpha$ that is the unique simple closed geodesic in the cover $X$. Also recall that $\widetilde \alpha$ is the only lift of $\overline \alpha$ in $\overline Y$, and that $\widetilde \alpha$ is a simple geodesic arc in $\overline Y$ with endpoints in $\partial \overline Y$. Since $\overline Y$ is a fundamental domain for the action of $K$, the restriction of $f$ to $\mathring{\overline Y}$ is a homeomorphism into $\mathbb{D}/K$. Therefore, $\mathring{\widetilde \alpha}$ also projects by a homeomorphism into $\mathbb{D}/K$ since $\mathring{\widetilde \alpha} \subset \mathring{\overline Y}$. Thus, $f(\widetilde \alpha)$ is an embedded loop in $\mathbb{D}/K$.

Now, since $\alpha \in K'$ we know that $\alpha' =s(\widetilde \alpha)$ is also a loop in $N=\mathbb{D}/K'$. But, $f(\widetilde \alpha) = u(s(\widetilde \alpha))$ so that if $s(\widetilde \alpha)$ is not an embedded loop in $N$, $f(\widetilde \alpha)$ cannot be an embedded loop in $\mathbb{D}/K$. That is, if $x_1$ and $x_2$ are two points of $\widetilde \alpha$ that are identified under the map $s$, then $f(x_1) = u(s(x_1)) = u(s(x_2)) = f(x_2)$ is a self intersection point of $f(\widetilde \alpha)$. Thus, $\alpha'$ must be an embedded loop in $N = \mathbb{D}/K'$. 

\end{proof}

We have now proved Theorem 1.1, and actually have proved the following stronger result. 

\begin{theorem}
Let $\Sigma$ be a closed surface of negative Euler characteristic, endowed with the standard metric. For every closed geodesic $\alpha$ in $\Sigma$, there exists a finite cover $X_{\alpha}$ of $\Sigma$ in which $\alpha$ lifts to an embedded loop. The index of the cover is bounded by $16.2\ell$, where $\ell$ is the length of the geodesic $\alpha$. If $\alpha$ is the image in $\Sigma$ of a geodesic line in our tessellation of $\D$, the index of the cover is bounded by $3.1\,\ell$.
\end{theorem}

\begin{proof} 
Our finite cover $X_{\alpha}$ is the cover $N = \D/K' $ of Lemma 5.3, and the lift of $\alpha$ that is an embedded geodesic loop is $\alpha'$. $\pi_1(N) = K'$ and by Lemma 5.2 and Corollary 4.4,  $[\pi_1(\Sigma)\!:\!K'] < k < 16.2 \ell$. Thus, the index of $X_{\alpha}$ as a cover of $\Sigma$ is bounded by $16.2\,\ell$. If $\alpha$ is the image of a line in our tiling of $\D$, $k < 3.1\, \ell$ and the result follows.

\end{proof} 

We have proved Theorem 5.4 for any closed surface of negative Euler characteristic with the standard metric. However, the theorem also holds for compact surfaces with boundary of negative Euler characteristic. In \cite{Scott}, Peter Scott showed that for the closed case, $\Sigma$ can be tiled by regular, right-angled pentagons, and thus proved that $\pi_1(\Sigma) < \pi_1(F) < \Gamma$. It turns out that compact surfaces with boundary, of negative Euler characteristic, can also be tiled by regular, right-angled pentagons. We endow such a surface $\Sigma$ with the metric obtained through this tiling by pentagons, and call this the standard metric as well. With the standard metric, the universal cover of $\Sigma$ is a convex, non-compact polygon in $\D$, which we call $\widetilde \Sigma$. Therefore, $\pi_1(\Sigma)$ acts on $\widetilde \Sigma$ by isometries, but these isometries extend to isometries of $\D$.  Then considering $\pi_1(\Sigma)$ as a group of isometries of $\D$, we have that $\pi_1(\Sigma) <\Gamma$. Once we have this result, the rest of the proof follows exactly as in the closed case.

\section{Hyperbolic Surface Groups are Residually Finite}
\noindent {\bf Definition.} A group $G$ is said to be \emph{residually finite (RF)} if for every non-trivial element $g \in G$, there is a subgroup $G'$ of finite index in $G$ that does not contain $g$.

\vspace{.2 in}

Let $\Sigma$ be a compact surface. From \cite{Scott}, we know that $\pi_1(\Sigma)$ is RF. We will quantify this result by proving the following theorem, which is a slightly stronger result than Theorem 1.2.

\begin{theorem}
Let $\Sigma$ be a compact surface Euler characteristic endowed with the standard metric. For any $\alpha \in \pi_1(\Sigma)-\{id\}$, there exists a subgroup $H'$ of $\pi_1(\Sigma)$, such that $\alpha \notin H'$. Additionally, $[\pi_1(\Sigma)\!:\!H'] < 32.3\,\ell$, where $\ell$ is the length of the geodesic $\alpha$. If $\alpha$ is the image of a geodesic line in our tessellation of $\D$, $[\pi_1(\Sigma)\!:\!H'] < 6.2\,\ell$.
\end{theorem}

\begin{proof}
Let $\alpha \in \pi_1(\Sigma) - {id}$. Using the same notation as in Section 5, we let $\widetilde \alpha$ be one lift of $\overline \alpha$ to $\D$. We will now show that we can lift $Y$ to $\D$ so that the result is a convex union of $k$ pentagons. The convexity of the lift is crucial since we will want to apply the Poincar\'{e} Polygon Theorem to prove the result above.

Let $Z$ be the set defined in the proof of Theorem 4.3. We lift $Z$ to a lift in $\D$ that contains $\widetilde \alpha$. Recall that in the proof of Theorem 4.3 we lifted $Z$ so that it opened along the geodesics $\beta_i$.  Instead we lift $Z$ so that it opens along the geodesic line of our tessellation of $X$ containing the basepoint of $\overline \alpha$ and shown on the left in Figure~7(A) below. Call this set $\overline Z \subset \D$ and lift every pentagon of $Y$ to its lift that lies in $\overline Z$. The result is a convex, connected union of $k$ pentagons which we will call $\overline Y$ (see Figure~7(B)).

\begin{figure}[h]
\centering
\subfloat[]{\label{} \includegraphics[trim = 1in 3.5in 1in 4.25in, clip=true, totalheight=0.14\textheight]{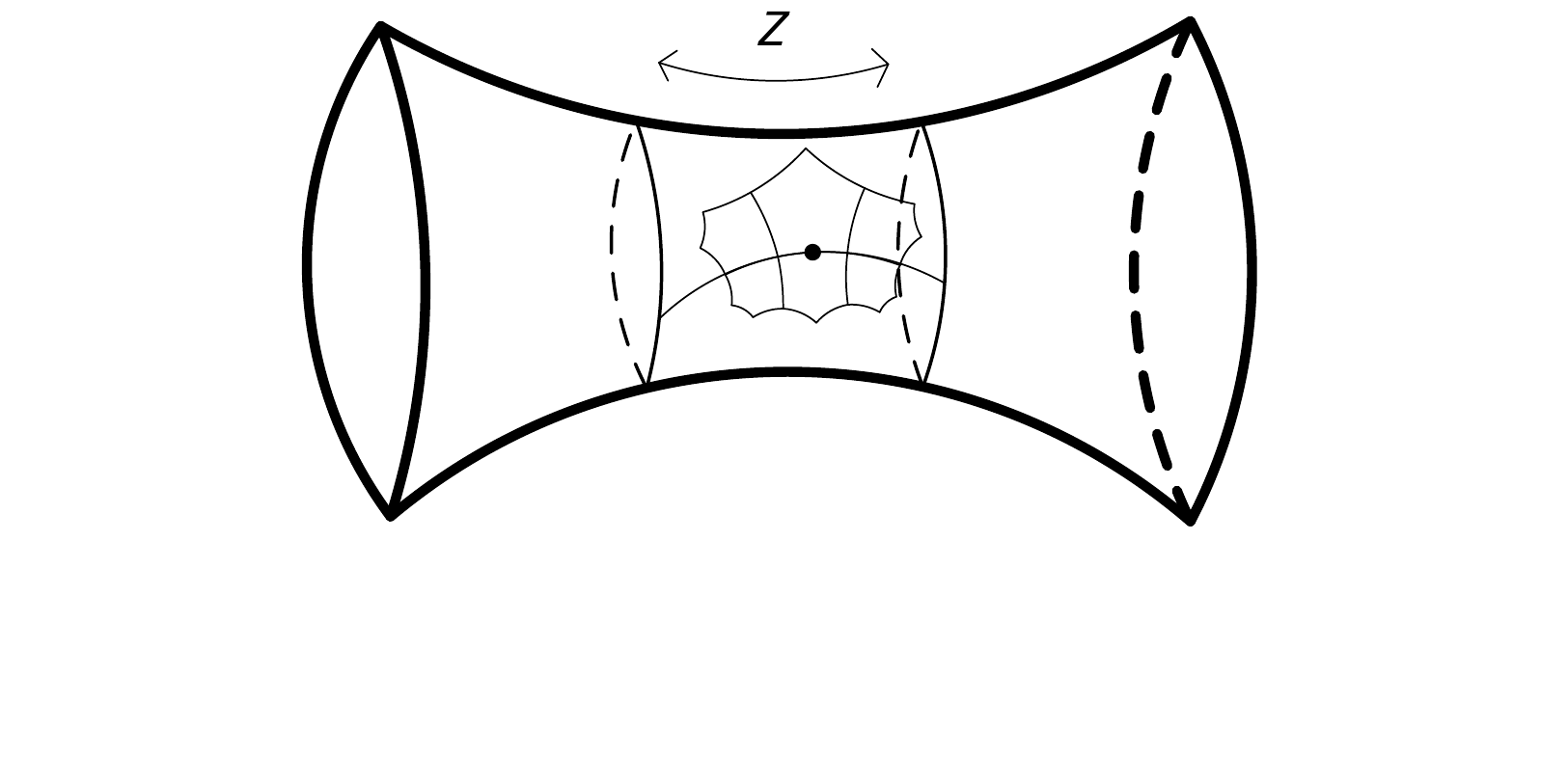}}
\subfloat[]{\label{} \includegraphics[trim = .5in 3.5in .5in 3.5in, clip=true, totalheight=0.15\textheight]{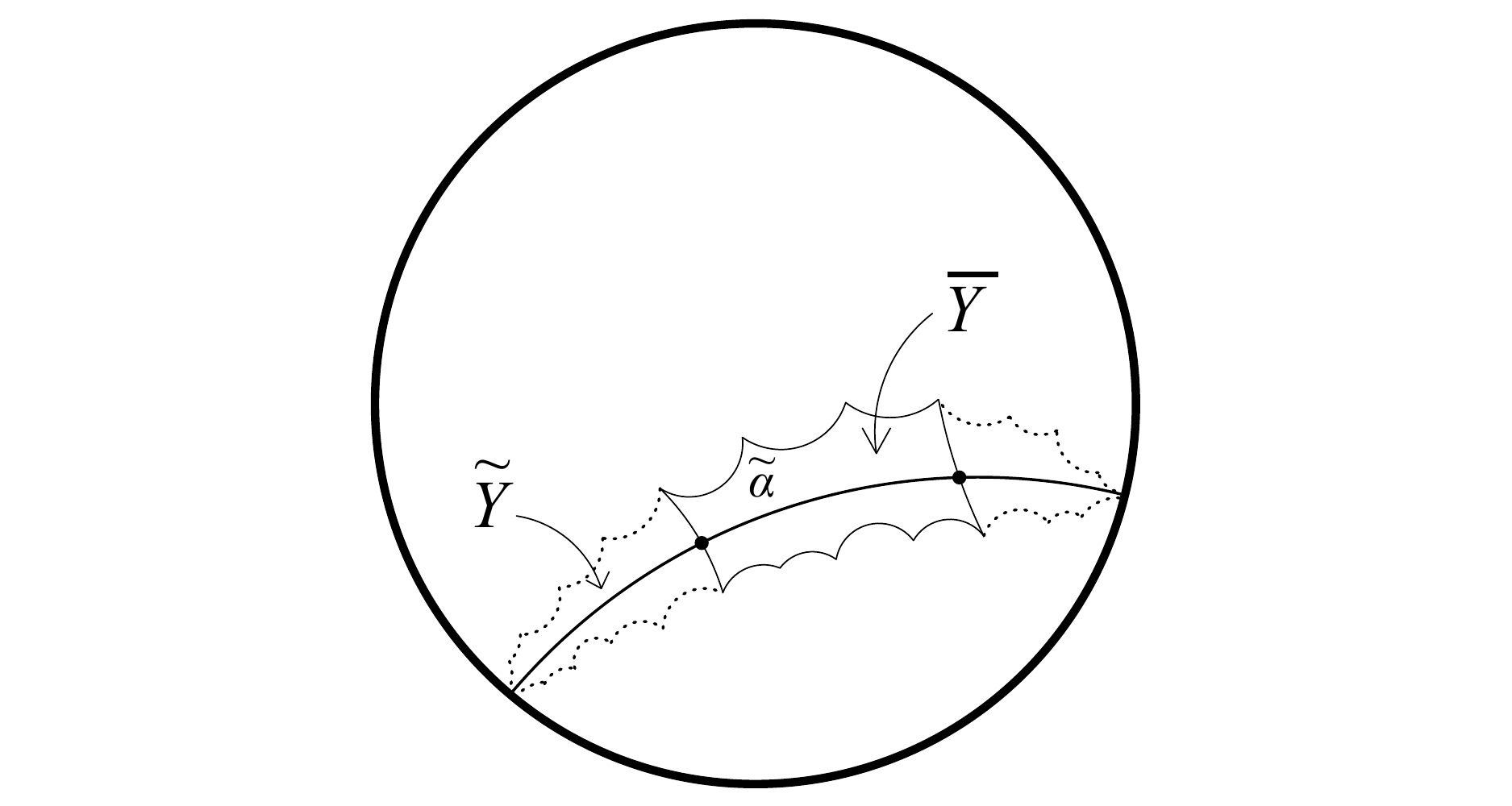}}
\caption{}
\end{figure}

Let $\widetilde \alpha_1$ be a lift of $\overline \alpha$ that shares an endpoint with $\widetilde \alpha$, and let $\overline Y_1$ be the lift of $Y$ containing $\widetilde \alpha_1$. Then $Y' = \overline Y \cup \overline Y_1$ is a convex union of $2k$ pentagons in $\D$, such that one endpoint of $\widetilde \alpha$ is contained in the interior of $Y'$.

Let $H$ be the group of isometries of $\D$ generated by reflections in the sides of $Y'$. Then $H < \Gamma$, and $Y'$ is a fundamental domain for the action of $H$ on $\D$ by the Poincar\'{e} Polygon Theorem \cite{Hubbard}. Since $Y'$ contains $2k$ pentagons, $[\Gamma\!:\!H] = 2k$. Letting $b: \D \longrightarrow \D/H$ be the covering map, we then have that the restriction of $b$ to $\mathring{Y'}$ is a homeomorphism onto its image in $X'$. Thus, $b(\widetilde \alpha)$ is not a loop in $X'$, and $\alpha \notin H$.

Now, let $H' = H \cap \pi_1(\Sigma)$. Then, $\alpha \notin H'$ and $[\pi_1(\Sigma)\!:\!H'] \leq [\Gamma\!:\!H] = 2k$.  The result follows from Corollary 4.4 and Corollary 4.5. 
 
\end{proof}

One should note that the bound in Theorem 6.1 can certainly be improved. Instead of adding one more set of $k$ pentagons to $\overline Y$, we could have simply added a small number of pentagons in order to encapsulate one endpoint of $\widetilde \alpha$ and retain convexity. We used the method above to simplify the proof.

\section {Hyperbolic Surface Groups are LERF}
{\bf Definition.} A group $G$ with a subgroup $S$ is \emph{S-residually finite} if for any element $g$ of $G-S$, there is a subgroup $G'$ of finite index in $G$ which contains $S$ but not $g$. A group $G$ is called \emph{locally extended residually finite (LERF)} if $G$ is S-residually finite for every {\bf finitely generated} subgroup $S$ of $G$.

\vspace{.15 in}

Let $\Sigma$ be a compact surface. From \cite{Scott}, we know that $\pi_1(\Sigma)$ is LERF, and we will attempt to quantify this result. Just as for Theorem 5.4, we will prove the result in the closed case, and then will see that the compact with boundary case follows immediately.

Let $\Sigma$ be a closed surface of negative Euler characteristic with the standard metric, and let $S$ be a finitely generated subgroup of $\pi_1(\Sigma)$ with $g \in \pi_1(\Sigma)-S$. If $S$ is a finite index subgroup, then $S$ itself is the required subgroup for the LERF condition. Thus, we will be interested in the case where $S$ is a finitely generated, infinite index subgroup of $\pi_1(\Sigma)$. 

Let $X$ be the cover of $\Sigma$ corresponding to such a subgroup $S$. If we pull back the standard metric on $\Sigma$ to $X$, then $X$ is a noncompact hyperbolic surface of finite type. Thus, $\pi_1(X) \cong S$ is a free group of rank $n$.

Let $\gamma \in \pi_1(\Sigma) - S$.  As per our convention, we also let $\gamma$ be the unique geodesic representative in this homotopy class and let $\ell_\gamma$ be the length of $\gamma$. Let $\widetilde \gamma$ be a lift of $\gamma$ to $X$. Since $\gamma \notin S$, $\widetilde \gamma$ is a (non-closed) geodesic path in $X$. 

Let $C(X)$ be the convex core of $X$, that is, $C(X)$ is the smallest, closed, convex subsurface of $X$ with geodesic boundary, such that $i: C(X) \longrightarrow X$ is a homotopy equivalence. Choose the basepoint, $x_0$, of $X$ to be in $C(X)$, and let $\alpha_1, \dots, \alpha_m$ be the geodesic boundary components of $C(X)$ of lengths $\ell_1, \dots, \ell_m$, respectively.

We will quantify Peter Scott's LERF theorem by proving the following result.

\begin{theorem}
Let $\Sigma$ be a compact surface of negative Euler characteristic with the standard metric. If $S$ is an infinite index, finitely generated subgroup of $\pi_1(\Sigma)$ and $\gamma \in \pi_1(\Sigma) - S$ as described above, then there exists a finite index subgroup $K'$ of $\pi_1(\Sigma)$, such that $S \subseteq K'$ and $\gamma \notin K'$. When the rank of $S$ is $n \geq 2$, the index of $K'$ in $\pi_1(\Sigma)$ can be bounded as follows. If $\widetilde \gamma \subset C(X)$,

\begin{equation}
[\pi_1(\Sigma)\!:\!K'] < 4n - 4 + 8.1 (\ell_1 + \cdots + \ell_m)
\end{equation}

\noindent and if $\widetilde\gamma \not\subset C(X)$,

\begin{equation}
[\pi_1(\Sigma)\!:\!K'] < 4n - 4 + \frac{2 \, \sinh \left[(\ell_\gamma/e_P \, + \, 2) \, d_0\right]}{\pi}\,(\ell_1 + \ell_2 + \cdots + \ell_m),
\end{equation}

\vspace{.15 in}

\noindent where $e_P$ is the length of the edges and $d_0$ is the diameter in a regular, right-angled, hyperbolic pentagon calculated in Section 4. If $\alpha_j$ is the image of a line in our tessellation of $\D$ for some j, the coefficient of $\ell_j$ can be improved to $1.6$ instead of $8.1$ in equation \emph{(1)}.

\vspace{.1 in}
\noindent In the case where the rank of $S$ is $n =1$, we must double the coefficients of the $\ell_1$ in equations \emph{(1)} and \emph{(2)} and we arrive at the following bounds:

\noindent If $\widetilde \gamma \subset C(X)$,

\begin{equation}
[\pi_1(\Sigma)\!:\!K'] < 16.2 \, \ell_1
\end{equation}

\noindent and if $\widetilde\gamma \not\subset C(X)$,

\begin{equation}
[\pi_1(\Sigma)\!:\!K'] < \frac{4 \, \sinh \left[(\ell_\gamma/e_P \, + \, 2) \, d_0\right]}{\pi}\, \ell_1. 
\end{equation}

\end{theorem}

\vspace{.2 in}

\begin{proof}
We handle the case where $n \geq 2$ and will elaborate on the case where $n = 1$ at the end of the proof. 

We will be interested in extending $C(X)$ at the boundary components $\alpha_i$ in order to obtain a convex union of pentagons containing $\widetilde \gamma$ in its interior. Then we will apply the same methods we used to prove the RF case. 

Let $S_i$ be the set of pentagons in the tiling of $X$ whose intersection with $\alpha_i$ is non-empty. Let $Y_i$ be the one sided convexification, i.e. the convexification of the side of $\alpha_i$ in $X - C(X)$, of each set $S_i$, obtained by the procedure in Section 3 and shown in Figure~8 below. 

\begin{figure}[h]
\centering
 \includegraphics[trim = .5in 2.5in .45in 2.5in, clip=true, totalheight=0.15\textheight]{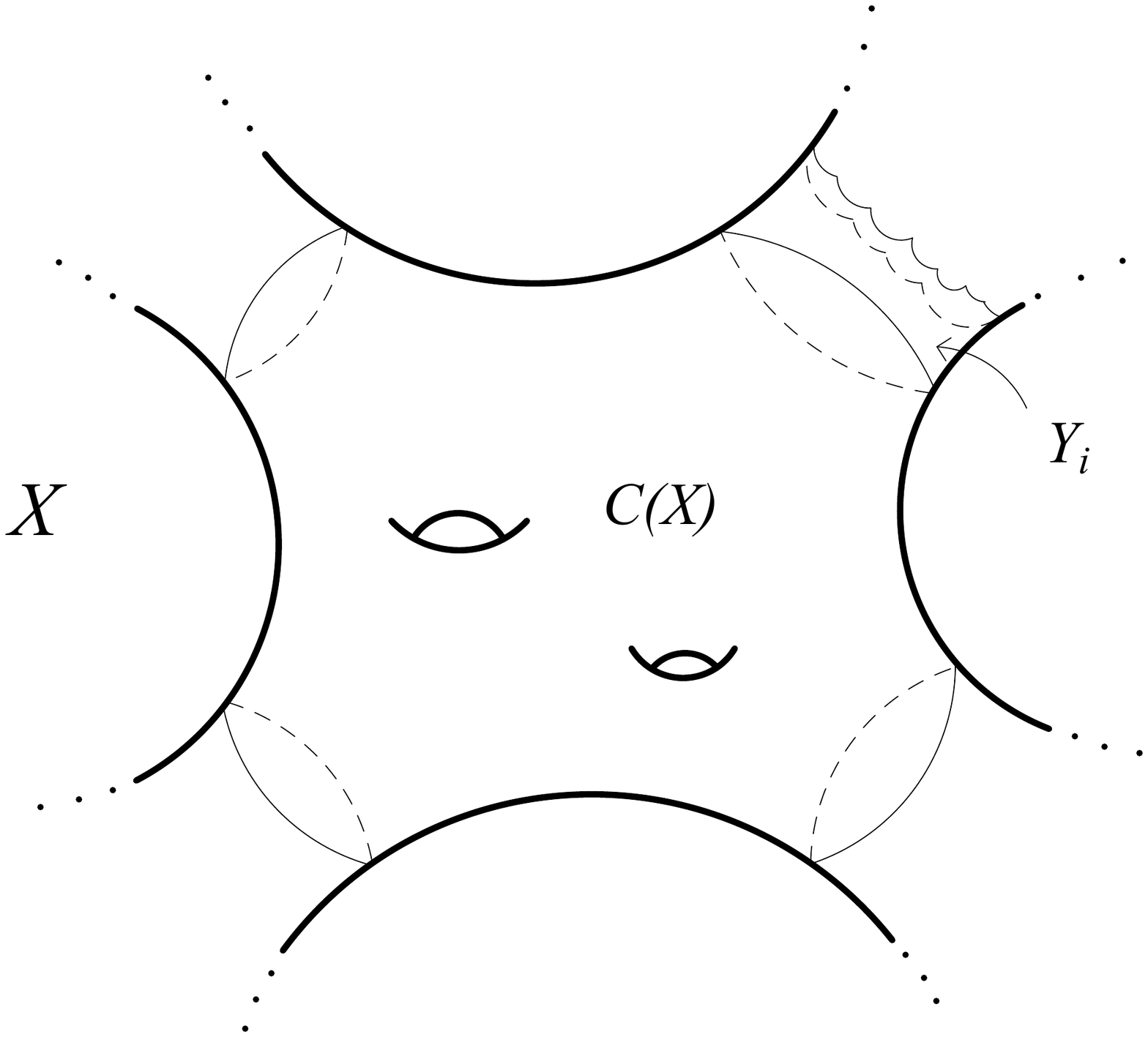}
 \caption{}
\end{figure}

Note that though $X$ may be unorientable, it makes sense to talk about the side of $\alpha_i$ in $C(X)$ and the side of $\alpha_i$ in $X-C(X)$ because $\alpha_i$ admits a bi-collar neighborhood in $X$ that is orientable.

\vspace{.15 in}

\underline{Case 1:} Suppose $\widetilde \gamma$ is completely contained in $C(X)$. Then $Y_i \subset C(X) \cup Z_i$, where $Z_i = \left\{x \in \overline {X - C(X)} : d(x, \alpha_i) \leq 2d_0\right \}$. From the results of Section 4, we know $Area(Z_i) = \sinh(2 \, d_0) \, \ell_i$.

Since $\pi_1(C(X)) \cong \pi_1(X) \cong S$, we know $\chi(C(X)) = 1-n$. Then by the Gauss-Bonnet Theorem \cite{Lee}, we have that $Area(C(X))= 2\pi(n-1)$.

Let $Y = C(X) \cup Y_1 \cup \cdots \cup Y_m$. Then $Y \subset C(X) \cup Z_1 \cup \cdots \cup Z_m$, and $Y$ is a convex union of $k'$ pentagons, where $k' \frac{\pi}{2} = Area(Y) \leq 2\pi(n-1) + \sinh(2 \, d_0) \, \ell_1 + \cdots + \sinh(2 \, d_0) \, \ell_m$. Therefore,
$$k' < 4n - 4 + 8.1 (\ell_1 + \cdots + \ell_m).$$

\vspace{.1 in}

\noindent Note: If any $\alpha_i$ is the image of a line in our tessellation of $\D$, then $S_i$ is automatically convex, so that $Y_i = S_i$. Therefore, $Y_i \subset C( X) \cup Z^* _ i$ where $Z^*_i = \{ x \in \overline{X- C(X)} : d(x, \alpha_i) \leq d_0 \}$, and $Area(Z^*_i) = \sinh (d_0) \, \ell_i \approx 1.55 \, \ell_i$. This gives us the improvement on the bound stated in Theorem 7.1. 

Let $\widetilde Y \subset \D$ be the set of all lifts of the pentagons in $Y$. As before, let $R = \left\{R_{y_i}\!: y_i \;\text{is a side of}\; \widetilde Y \right\}$ be the set of isometries of $\D$ consisting of reflections in the sides, $y_i$ of $\widetilde Y$. Then $\widetilde Y$ is a fundamental domain for the action of $\langle R \rangle$ on $\D$, and $\langle R \rangle < \Gamma$.

Let $W$ be a fundamental domain for the action of $S$ on $\D$ so that $W/S = X$ and $W$ is a union of pentagons in our tessellation of $\D$. Lift each pentagon of $Y$ to one of its lifts so that the result is a connected union of $k'$ pentagons contained in $W$. We call this union of $k'$ pentagons $\overline Y$.

Let $K = \langle R, S \rangle$, and let $X' = \D/K$. The remainder of the proof follows the arguments of Section 5 very closely. By an extension of the reasoning in Section 5, we have that $K = \langle R, S \rangle = \langle R \rangle \rtimes_\phi S$, and $\overline Y$ is a fundamental domain for the action of $K$ on $\D$. Thus, $[\Gamma\!:\!K] = k'$. Since $\widetilde \gamma$ is contained in $C(X) \subset \mathring Y$, all lifts of $\widetilde \gamma$ are contained in the interior of $\overline Y$. The image of $\widetilde \gamma$ is, therefore, not a loop in $X' = \D/K$ and $\gamma \notin K$.

Letting $K' = K \cap \pi_1(\Sigma)$, we have that $S \subset K'$, $\gamma \notin K'$ and
$$[\pi_1(\Sigma)\!:\!K'] \leq [\Gamma\!:\!K] = k' < 4n - 4 + 8.1 (\ell_1 + \cdots + \ell_m).$$

{\bf Explanation of equation (3):} If the rank of $S$ is $n =1$, the cover $X$ corresponding to $S$ is an open annulus or an open M\"{o}bius band as described in Section 2. In this case, $C(X) = \alpha_1$, and $\gamma \subseteq C(X)$. The analog of the sets $Z_i$ defined above is the set $Z= \left\{x \in \overline {X - C(X)} : d(x, \alpha_1) \leq 2d_0\right \}$ = $\left\{ x \in X : d(x, \alpha_1) \leq 2d_0 \right\}$. We know from the calculation in Theorem 4.3 that $Area(Z) = 2 \sinh (2\,d_0) \, \ell_1$, explaining the doubling of the coefficient of $\ell_1$.

\vspace{.05 in}

\underline{Case 2:} Suppose that $\widetilde \gamma$ is not completely contained in $C(X)$. Recall that we chose the basepoint of $X$ to be in $C(X)$ so that one endpoint of $\widetilde \gamma$ is $x_0 \in C(X)$. Then $\widetilde \gamma$ crosses a boundary component of $C(X)$, say $\alpha_1$, at some point and enters $X- C(X)$. Of course $\widetilde \gamma$ may extend past $Y_1$, and thus may not be contained in the convex space $Y$. Thus, we will add pentagons to $Y_1$ until we have encapsulated the portion of the curve $\widetilde \gamma$ in the non-compact region bounded by $\alpha_1$. We repeat this procedure for each $Y_i$ so that we have encapsulated all of $\widetilde \gamma$ in a larger convex union of pentagons and then apply the same method as in Case 1. 

Let $\partial Y_1$ be the portion of the boundary of $Y_1$ contained in $X- C(X)$. Let $U_1$ be the set obtained from $Y_1$ by adding all pentagons in the tessellation of $X$ that intersect $\partial Y_1$. Let $\partial U_1$ be the portion of the boundary of $U_1$ contained in $X-C(X)$. Since $Y_1$ is convex along $\partial Y_1$, $U_1$ will be convex along $\partial U_1$. When we extend $Y_1$ in this fashion we say that we have added \emph{one layer of pentagons} to $Y_1$. If we repeat this procedure for $U_1$, we say that we have added \emph{two layers of pentagons} to $Y_1$, and so on.

It is not hard to see that $d(\partial Y_1, \partial U_1) \geq e_P$, where $e_P \approx 1.062$ is the length of an edge of a pentagon in our tiling. Now, recall that $\ell_\gamma$ is the length of $\gamma$, and hence, the length of $\widetilde \gamma$. Let $Y_i'$ be the set obtained by adding $\ell_\gamma/e_P$ layers of pentagons to $Y_i$, for $i = 1, \dots, m$. Then  $Y' = C(X) \cup Y_1' \cup \cdots \cup Y_m'$ is a convex extension of $Y$, and we can ensure that $\widetilde \gamma$ is contained in the interior of $Y'$.

Let $Z_i' = \left\{x \in \overline{X-C(X)} : d(x, \alpha_i) \leq (\ell_\gamma/e_P + 2) \,d_0 \right\}$. Then, $Y_i' \subset C(X) \cup Z_i'$, and we have that $Area(Z_i') = \sinh [(\ell_\gamma/e_P \,+ \,2) \, d_0] \, \ell_i$. It then follows that if $Y'$ consists of $k''$ pentagons,

$$k'' < 4n - 4 + \frac{2 \, \sinh [(\ell_\gamma/e_P \,+ \,2) \, d_0]}{\pi}\,(\ell_1 + \ell_2 + \cdots \ell_m).$$

We follow the proof of Case 1 replacing the set $Y$ with $Y'$ to obtain a subgroup $K'$, such that $S \subset K'$ and $\gamma \notin K'$. In this case,

$$[\pi_1(\Sigma)\!:\!K'] \leq [\Gamma\!:\!K] = k'' < 4n - 4 + \frac{2 \, \sinh \left[(\ell_\gamma/e_P \, + \, 2) \, d_0\right]}{\pi}\,(\ell_1 + \ell_2 + \cdots + \ell_m).$$

{\bf Explanation of equation (4):} Again, in the case where the rank of $S$ is $n=1$, the analog of the sets $Z_i'$ is $Z' = \left\{x \in \overline{X-C(X)} : d(x, \alpha_1) \leq (\ell_\gamma/e_P + 2) \,d_0 \right\} = \left\{x \in X : d(x, \alpha_1) \leq (\ell_\gamma/e_P + 2) \,d_0 \right\}$ so that $Area(Z') = 2 \sinh [(\ell_\gamma/e_P \,+ \,2) \, d_0] \, \ell_1$, explaining the doubling of the coefficient.  

\end{proof}

We can now prove the case where $\Sigma$ is a compact hyperbolic surface with geodesic boundary. We know that the universal cover of $\Sigma$ is a convex, non-compact polygon in $\D$, which we call $\widetilde \Sigma$. By our comments after Theorem 5.4, we also know that $\pi_1(\Sigma) < \Gamma$ when we consider $\pi_1(\Sigma)$ as a group of isometries of $\D$. 

At each boundary component of $\Sigma$ we will glue in a non-compact region, as in Figure~9 below. We call this new surface $\Sigma'$. 

\begin{figure}[h]
\centering
 \includegraphics[trim = .5in 2.1in .5in 2in, clip=true, totalheight=0.15\textheight]{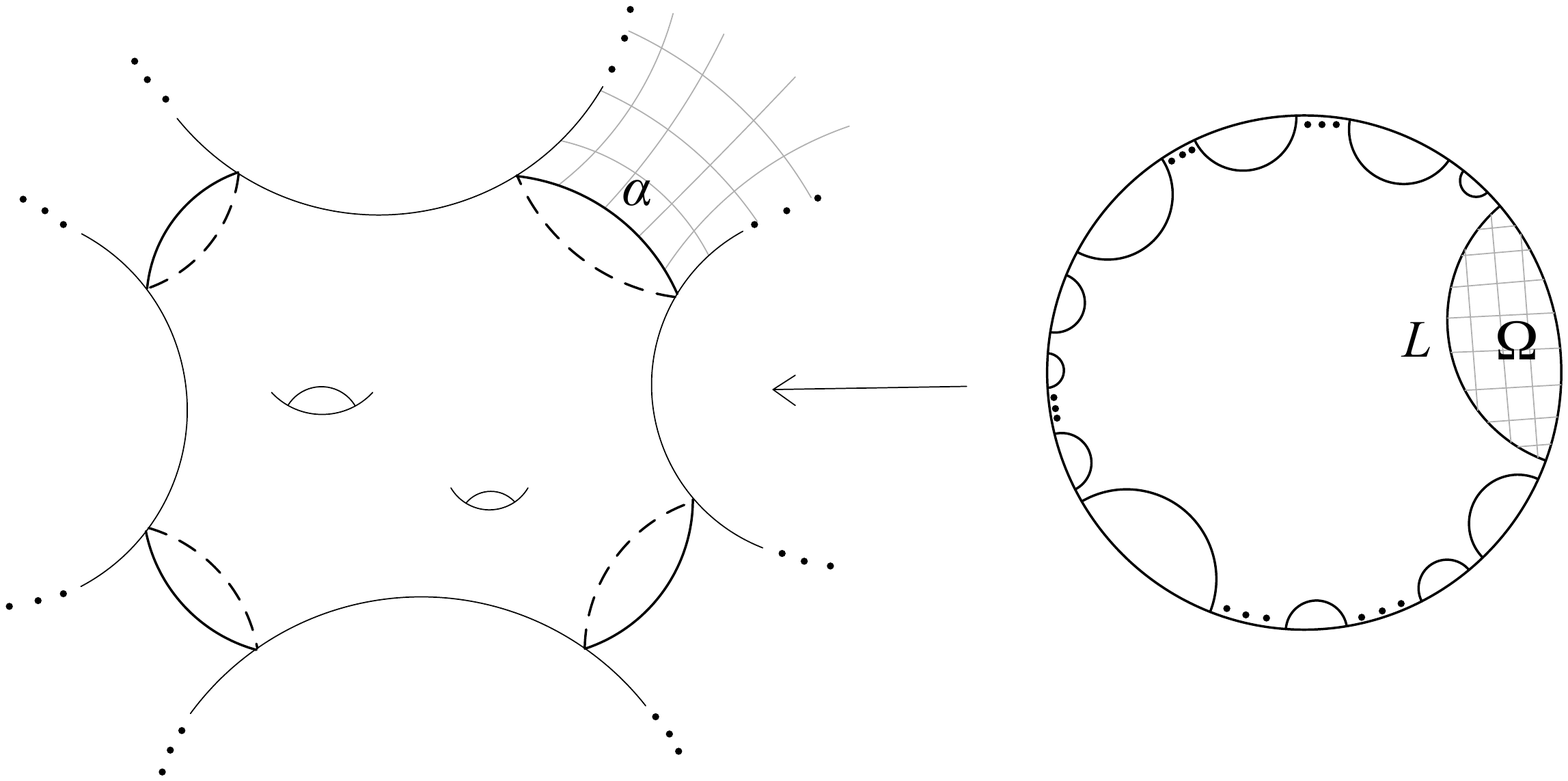}
 \caption{}
\end{figure}

In $\D$, one such gluing along a boundary component, $\alpha$, corresponds to the gluing of the region $\Omega$ to $\widetilde \Sigma$, as  in the figure above. The gluing occurs along the geodesic line $L$, whose image in $\Sigma$ is $\alpha$. Most importantly, $\pi_1(\Sigma) \cong \pi_1(\Sigma')$ and $\pi_1(\Sigma') < \Gamma$. 

Now we can follow the proof of the closed case for $\pi_1(\Sigma')$. Thus, $\pi_1(\Sigma')$ is LERF, and $\pi_1(\Sigma)$ is, therefore, LERF since the groups are isomorphic. The same bounds stated in Theorem 7.1 will hold.

\end{document}